\documentclass[a4paper,oneside]{article}

\usepackage{amsmath}
\usepackage{amsthm,amssymb}
\usepackage{xypic}
\usepackage{graphicx}
\usepackage[english]{babel}

\newcommand{\sol}{\mathrm{Sol}} 
\newcommand{\isom}{\operatorname{Isom}} 
\newcommand{\glz}{\operatorname{Gl_2(\z)}} 
\newcommand{\tors}{\mathrm{Tors}} 
\newcommand{\ind}{\mathrm{ind}} 

\newcommand{\sn}{\operatorname{sn}}

\newcommand{\p}[0]{{\mathbb P}}
\newcommand{\z}[0]{{\mathbb Z}}
\newcommand{\n}[0]{{\mathbb N}}
\newcommand{\q}[0]{{\mathbb Q}}
\let \cedilla=\c  
\renewcommand{\c}[0]{{\mathbb C}}
\renewcommand{\r}{\mathbb{R}}
\newcommand{\R}{\mathbb{R}}

\newcommand{\norm}[1]{\left\| #1 \right\|} 
\newcommand{\st}{~|~} 

\newtheorem{thm}
{Theorem}%
\numberwithin{thm}{section}%

\newtheorem{prop}
{Proposition}
\numberwithin{prop}{section}%

\newtheorem{lem}
{Lemma}
\numberwithin{lem}{section}%

\newtheorem{cor}
{Corollary}
\numberwithin{cor}{section}%

{Conjecture}
\numberwithin{conj}{section}%

\theoremstyle{remark}
\newtheorem{rem}
{Remark}
\numberwithin{rem}{section}%

\theoremstyle{definition}
\newtheorem{dfn}
{Definition}
\numberwithin{dfn}{section}%

\newtheorem{nota}
{Notation}
\numberwithin{nota}{section}%

\begin{document}

\title{Do uniruled six-manifolds contain $\sol$ Lagrangian submanifolds?}
\author{Fr\'ed\'eric Mangolte \and Jean-Yves Welschinger} 

\date{January 16, 2010}

\maketitle

\begin{abstract}
We prove using symplectic field theory that if the suspension of a hyperbolic diffeomorphism of the two-torus Lagrangian embeds in a closed uniruled symplectic six-manifold, then its image contains the boundary of a symplectic disc with vanishing Maslov index. This prevents such a Lagrangian submanifold to be monotone, for instance the real locus of a smooth real Fano manifold. It also prevents any $\sol$ manifold to be in the real locus of an orientable real  Del Pezzo fibration over a curve, confirming an expectation of J.~Koll\'ar. Finally, it constraints Hamiltonian diffeomorphisms of uniruled symplectic four-manifolds. 
\end{abstract}
\section*{Introduction}

 Complex projective uniruled manifolds play a special r\^ole in algebraic geometry, these are the manifolds of special type in the sense of Mori. What can be the topology of the real locus of such a manifold when defined over $\r$? This natural question has a symplectic counterpart. What can be the topology of  Lagrangian submanifolds of such uniruled manifolds? Uniruled manifolds of dimension two are rational or ruled surfaces. Comessatti proved in \cite{com} that no orientable component of the real locus of such a surface can have negative Euler characteristic. Actually, closed symplectic four-manifolds with $b_2^+ = 1$ cannot contain any orientable Lagrangian submanifold with negative Euler characteristic. By the way, it is proved in \cite{wels} that even the unit cotangent bundle of an orientable hyperbolic surface does not embed as a hypersurface of contact type of a uniruled symplectic four-manifold. 

In complex dimension three, a great piece of work was done by Koll\'ar \cite{koI,koII,koIII,koIV} in order to carry out Mori's minimal model
program (MMP) over $\r$ for uniruled manifolds. Roughly, the upshot \cite{ko-nash} is that up to connected sums with $\r\p^3$ or $S^2\times S^1$ and modulo a finite number of closed three-manifolds, the orientable real uniruled three-manifolds are Seifert fibered spaces or connected sums of lens spaces. This result however depends on two expectations. The first one is that closed hyperbolic manifolds cannot appear. The second one is that closed $\sol$ manifolds cannot appear. Quickly, this first expectation was confirmed by Viterbo and Eliashberg (\cite{sftVit,KhBourbaki,EGH}). Namely, a closed uniruled symplectic manifold of dimension greater than four cannot contain a closed Lagrangian submanifold with negative curvature. The proof of Eliashberg uses symplectic field theory (SFT), which appears to be a very powerful tool to tackle this question.

The aim of this paper is to prove the second one, using  SFT as well, at least as far as the precise expectation of Koll\'ar is concerned. We unfortunately could not prove such a general result as Viterbo-Eliashbeg's one, but proved the following (see Theorem~\ref{thm.main}).
Let $L \subset (X,\omega)$ be a closed Lagrangian submanifold homeomorphic to the suspension of a hyperbolic diffeomorphism of the two-torus, where $(X,\omega)$ is a closed symplectic uniruled six-manifold. Then $X$ contains a symplectic disc of vanishing Maslov index and with boundary on $L$, non-trivial in $H_1(L;\q)$. This prevents $L$ from being monotone, for instance the real locus of a smooth Fano manifold. It also actually prevents any $\sol$ manifold to be in the real locus of a projective three-manifold fibered over a curve with rational fibers, at least provided this real locus be orientable, see Corollary~\ref{cor.dp} and the discussion which follows. This was the actual problem raised by Koll\'ar in \cite[Remark~1.4]{koIV}. Finally, it implies that a Hamiltonian diffeomorphism
of a uniruled symplectic four-manifold which preserves some Lagrangian torus cannot restrict to a hyperbolic diffeomorphism of the torus, 
see Corollary~\ref{cor.torus}.
Our approach, which uses SFT, requires some understanding of the geodesic flow of $\sol$ manifolds, namely the Morse indices of their closed geodesics. The first part of this paper is thus devoted to a study of $\sol$ manifolds and their closed geodesics. The second part is devoted to the proof of our main result. Note that the converse problem remains puzzling. What is the simplest real projective manifold which contains a hyperbolic component? What is the simplest real projective manifold which contains a $\sol$ component? 
Recall that every closed orientable three-manifold modeled on any of the six remaining three-dimensional geometries embeds in the real locus of a projective uniruled manifold \cite{hm1,hm2}. Note also that
in the case of the projective space, the absence of orientable Sol Lagrangian submanifolds follows from Theorem 14.1 of \cite{Fuk}. 
Moreover, in this paper Kenji Fukaya remarks that his methods may extend to uniruled manifolds as well. \\

\textbf{Acknowledgments:}
This work was supported by the French Agence nationale de la recherche, reference ANR-08-BLAN-0291-02. The second author acknowledges Gabriel Paternain for fruitful discussions about $\sol$ manifolds.

\section{$\sol$-geometry}

\subsection{The group $\sol$}

The group $\r$ of real numbers acts on the abelian group $\r^2$ by
$$
\begin{array}{clc}
 \r\times\r^2 &\to &\r^2 \\
\left(z,(x,y)\right)  &\mapsto &\left(e^z x,e^{-z} y \right) \;.    
\end{array}
$$
The induced semidirect product is denoted by $\sol$, so that the group law of $\sol$ is given by
$$
\begin{array}{clc}
 \sol \times \sol &\to &\sol \\
\left((\alpha,\beta,\lambda),(x,y,z)\right)  &\mapsto &\left(e^\lambda x + \alpha,e^{-\lambda} y +\beta,z+\lambda\right) \;.    
\end{array}
$$
Let $K \cong \r^2$ denote the kernel of the surjective morphism $(x,y,z)\in \sol\mapsto z\in\r$, and let $e_1,e_2,e_3$ be the elements of $\sol$ of coordinates $(1,0,0)$, $(0,1,0)$, and $(0,0,1)$ respectively. 
The group $K$ coincides with the derived subgroup of $\sol$, as shows the relation
$$
[e_3,xe_1+ye_2]=(e-1)xe_1+(e^{-1}-1)ye_2\;.
$$

Denote by
$$
X:=e^z\frac{\partial}{\partial x}\quad,\quad Y:=e^{-z}\frac{\partial}{\partial y}\quad,\quad Z:=\frac{\partial}{\partial z}
$$
the left-invariant vector fields of $\sol$ which coincide with $\frac{\partial}{\partial x},\frac{\partial}{\partial y},\frac{\partial}{\partial z}$ at the origin. 
We provide $\sol$ with  the Riemannian metric and the orientation making $(X,Y,Z)$ direct orthonormal.
The space $\sol$ thus obtained is homogeneous, its isotropy group is isomorphic to the diedral group $D_4$ generated by the isometries 
$$
\rho\colon (x,y,z)\in \sol\mapsto (y,-x,-z)\in \sol
$$
and
$$
r_Y \colon(x,y,z)\in \sol\mapsto (-x,y,z)\in \sol
$$
see \cite[Lemmas 3.1 and 3.2]{tro}.
In particular, the isometries of $\sol$ preserve the horizontal foliation $\mathcal{F}:=\{\mathrm dz=0\}$, and act by isometries on its space of leaves $\r$. We denote by $P \colon \isom(\sol)\to \isom(\r)$ the surjective morphism thus defined.

\subsection{Geodesic flow on $\sol$}

Geodesics of $\sol$ have been determined in \cite{tro}, and divided into three types $A$, $B$, and $C$. 
Geodesics of type $A$ are the lines directed by the vector field $f_1:=\frac{X-Y}{\sqrt{2}}$ or $f_2:=\frac{X+Y}{\sqrt{2}}$; they are contained in the foliation $\mathcal{F}$ (whose leaves are minimal surfaces). 
Geodesics of type $B$ are geodesics contained in the totally geodesic hyperbolic foliations $\mathcal{H'}:=\{\mathrm dy=0\}$ or $\mathcal{H''}:=\{\mathrm dx=0\}$. 
Among  geodesics of type $B$, only those directed by the vector field $Z$ will play a r\^ole in this paper. 
Geodesics of type $C$ are contained in cylinders whose axes are geodesics of type~$A$. Along such a geodesic, the $z$-coordinate 
-corresponding to the axis directed by $e_3$- is then bounded between two values, these 
geodesics of type $C$ will not play a significant r\^ole in the sequel. 
The aim of the present paragraph is to calculate the linearization of the geodesic flow along a geodesic of type $A$ or~$B$.

Denote by $S^*\sol:=\left\{ (q,p)\in T^*\sol \st \norm{p}=1\right\}$ the unitary cotangent bundle of $\sol$ where the norm $\norm{\cdot}$ is the one induced by the fixed Riemannian metric on $\sol$.
Denote by $\xi^*$ the contact distribution of $S^*\sol$, it is  the kernel of the restriction of the Liouville form $p\,\mathrm dq$. 
Likewise, we denote by $S\sol$ the unitary tangent bundle of $\sol$, and by $\xi$ the distribution induced by the identification $\flat \colon S\sol \stackrel{\sim}{\longrightarrow}  S^*\sol$ given by the metric. 
The identification is defined in the basis $(e_1,e_2,e_3)$ by
$$
\begin{array}{lclc}
\flat \colon &T\sol &\longrightarrow& T^*\sol \\
&\left(x,y,z,\dot x,\dot y,\dot z\right)  &\longmapsto &\left(x,y,z,e^{-2z}\dot x,e^{2z}\dot y,\dot z\right) \;.    
\end{array}
$$
The Levi-Civita connection gives an orthogonal direct sum decomposition 
$$
\xi = \xi_h \oplus \xi_v\;,
$$
where $\xi_v$ is the space of elements of $\xi$ which are tangent to the fibers of $T\sol$, while $\xi_h$ is the orthogonal plane to $\xi_v$ given by the connection. The planes $\xi_h$ and $\xi_v$ are canonically isomorphic; if $v$ is a tangent vector to $\sol$ orthogonal to a geodesic, we will denote by $v$ its lift to $\xi_h$, and by $\dot v$ its lift to $\xi_v$, in order to distinguish them.

\subsubsection{Linearized flow along a geodesic of type $A$}

There are two families of geodesics of type $A$, those given by $f_2=\frac{X+Y}{\sqrt{2}}$, and those given by $f_1=\frac{X-Y}{\sqrt{2}}$. 
Since these families are exchanged by the isometry $r_X = \rho^2r_Y$, we restrict our study to the first family.

Let then $\gamma(t):=\gamma_0 + t \left( \frac{X+Y}{\sqrt{2}}\right)_{\mathbf{|}\gamma_0}$ be a geodesic of type $A$ and $\gamma'(t)=\frac{\mathrm d}{\mathrm d t}\gamma(t) = \left( \frac{X+Y}{\sqrt{2}}\right)_{\mathbf{|}\gamma(t)}$. 
The orthogonal plane to $\gamma'(t)$ in $T_{\gamma(t)}\sol$ is generated by $\left(X-Y\right)_{\mathbf{|}\gamma(t)}$ and $Z_{\mathbf{|}\gamma(t)}$. 
Hence $\xi_{\mathbf{|}\left(\gamma(t),\gamma'(t)\right)} = \langle X-Y,Z,\dot X - \dot Y, \dot Z\rangle$. 
Let 
$$
h_1:=\frac{X-Y}{\sqrt{2}}\quad,\quad h_2:= Z + \dot h_1\quad,\quad h_3:= Z + 2\dot h_1 \quad,\quad h_4:=h_1 - \dot Z\;.
$$

\begin{lem}\label{lem.sp}
Let $\gamma \colon t \in \r \longmapsto \gamma_0 + t \left( \frac{X+Y}{\sqrt{2}}\right)_{\mathbf{|}\gamma_0}\in \sol$ be a geodesic of type $A$, where $\gamma_0 \in \sol$. Then the canonical symplectic form $\flat^*(\mathrm d p \wedge \mathrm d q)$ on the contact distribution $\xi$ along $\gamma$ is given by 
$$
\mathrm d h_1\wedge \mathrm d h_2 + \mathrm d h_3\wedge \mathrm d h_4\;.
$$
\end{lem}

\begin{proof}
The pull-back of the Liouville form is
$$
\flat^*(p\,\mathrm dq)_{\mathbf{|}\left(x,y,z,\dot x,\dot y,\dot z\right)} =
 e^{-2z}\dot x\,\mathrm dx + e^{2z}\dot y\,\mathrm dy + \dot z\,\mathrm dz\;,
$$
so that the symplectic form writes

\begin{align}
&\flat^*(\mathrm dp\wedge \mathrm dq)_{\mathbf{|}\left(x,y,z,\dot x,\dot y,\dot z\right)}&\notag\\
&= \left( -2 e^{-2z}\dot x \mathrm dz\wedge \mathrm dx + e^{-2z}\mathrm d\dot x\wedge \mathrm dx\right)
+  \left( 2 e^{2z}\dot y \mathrm dz\wedge \mathrm dy + e^{2z}\mathrm d\dot y\wedge \mathrm dy\right)\notag\\
&\quad + \mathrm d \dot z\wedge \mathrm dz\tag{$*$}\label{formula.sp}\\
&=  \mathrm d \dot u\wedge \mathrm du +  \mathrm d \dot v\wedge \mathrm dv +  \mathrm d \dot z\wedge \mathrm dz + 2( \dot v\mathrm du + \dot u\mathrm dv)\wedge \mathrm dz\;,\notag
\end{align}
where $(u,v)$ are coordinates in the basis $(\frac{X-Y}{\sqrt{2}},\frac{X+Y}{\sqrt{2}})$, and $(\dot u,\dot v)$ are coordinates in the basis $(\frac{\dot X-\dot Y}{\sqrt{2}},\frac{\dot X+\dot Y}{\sqrt{2}})$. 
The restriction of this symplectic form to the distribution $\xi$ along our geodesic of type $A$ is 
$\mathrm d \dot u\wedge \mathrm du + \mathrm d \dot z\wedge \mathrm dz + 2 \mathrm du\wedge \mathrm dz$, since $\dot v \equiv 1$ and $\dot u \equiv 0$, and this eventually gives $\mathrm d h_1\wedge \mathrm d h_2 + \mathrm d h_3\wedge \mathrm d h_4$.
\end{proof}

\begin{prop}\label{prop.index.A}
Let $\gamma \colon t \in \r \longmapsto \gamma_0 + t \left( \frac{X+Y}{\sqrt{2}}\right)_{\mathbf{|}\gamma_0}\in \sol$ be a geodesic of type $A$, where $\gamma_0 \in \sol$. 
The linearization of the geodesic flow of $\sol$ along $\gamma$ restricted to the contact distribution $\xi$ has the following matrix in the basis $(h_1,h_2,h_3,h_4)$:
$$
\left[
\begin{matrix}
	1 & t 	& 0	& 0\cr
	0 & 1		& 0					& 0\cr
	0 & 0		& \cos(\sqrt 2 t)			& - \frac 1{\sqrt 2}\sin (\sqrt 2 t)\cr
	0 & 0		& \sqrt 2\sin (\sqrt 2 t)	& \cos(\sqrt 2 t)\cr
\end{matrix}
\right]\;.
$$
\end{prop}

\begin{proof}
The vector field $h_1$ is the restriction along $\gamma$ of a Killing field of $\sol$. 
Likewise, the vector field $t \left( \frac{X - Y}{\sqrt{2}}\right)_{\mathbf{|}\gamma(t)} + Z_{\mathbf{|}\gamma(t)}$ is the restriction to $\gamma$ of  a Killing field of $\sol$. 
We deduce from that the two first columns of the matrix. Without loss of generality, we can assume that $\gamma_0 = 0$, so that $\gamma \colon t\in\r \mapsto tf_2(0)\in\sol$. 
A geodesic of type $C$ which is close to $\gamma$ writes 
$$
\gamma_k(t)=u_k(t)f_1(k) + v_k(t)f_2(k) + z_k(t)e_3 \;.
$$
Then 
$$
\left(\frac{\partial}{\partial k}\gamma_k(t)\right)_{\mathbf{|}k = 0} 
= \left(\frac{\partial}{\partial k}u_k(t)\right)_{\mathbf{|}k = 0}f_1(0) 
+ v_0(t)\frac{\partial f_2}{\partial k}_{\mathbf{|}k = 0} 
+ \left(\frac{\partial}{\partial k}z_k(t)\right)_{\mathbf{|}k = 0}e_3
$$
since $u_0(t)\equiv 0$ and we consider only the normal part of vector fields. 

Now, with the notations of \cite[\S4.4]{tro}, $u_k(t) = d + \mu k\sn\left(\mu(t+\tau)-K\right)$, where, since we assume that $\gamma_k(0) \equiv 0$, either $d = \mu k$ and $\tau = 0$, or $d = 0$ and $\tau = -\frac K\mu$. 
In the first case, we get 
$\left(\frac{\partial}{\partial k}u_k(t)\right)_{\mathbf{|}k = 0} = \sqrt 2 \left(1 + \sin (\sqrt 2 t - \frac \pi 2)\right)$, 
while $v_0(t)=t$ and $\frac{\partial f_2}{\partial k}_{\mathbf{|}k = 0} = \frac{\partial \bar z}{\partial k}_{\mathbf{|}k = 0} f_1(0)$. 
Now, keeping these notations: $z_k(t) = \bar z + h(\mu t - K)$, so that
$\frac{\partial z_k}{\partial k}_{\mathbf{|}k = 0} = 
\frac{\partial \bar z}{\partial k}_{\mathbf{|}k = 0} + \cos(\sqrt 2 t - \frac \pi 2)$. Thus, the vector field 
$$
\frac{\partial \bar z}{\partial k}_{\mathbf{|}k = 0}(tf_1 + e_3) + \sqrt 2 \left( -\cos(\sqrt 2 t) + 1\right)f_1 + \sin(\sqrt 2 t)e_3
$$
along $\gamma$ is a Jacobi field. We deduce that 
$
\left(1  -\cos(\sqrt 2 t)\right)f_1 + \frac 1 {\sqrt 2}\sin(\sqrt 2 t)e_3
$
is Jacobi itself and then the fourth column of the matrix. In the second case, we get 
$\left(\frac{\partial}{\partial k}u_k(t)\right)_{\mathbf{|}k = 0} = -\sqrt 2\sin(\sqrt 2 t)$, while $\frac{\partial \bar z}{\partial k}_{\mathbf{|}k = 0} = 1$ and $\frac{\partial z_k}{\partial k}_{\mathbf{|}k = 0} = 
\frac{\partial \bar z}{\partial k}_{\mathbf{|}k = 0} - \cos(\sqrt 2 t)$. Hence, the vector field 
$$
\left(t - \sqrt 2\sin(\sqrt 2 t)\right)f_1 + \left(1  -\cos(\sqrt 2 t)\right)e_3
$$ 
along $\gamma$ is Jacobi, so that $\sqrt 2\sin(\sqrt 2 t)f_1 + \cos(\sqrt 2 t)e_3
$
 is Jacobi itself.
 \end{proof}

\subsubsection{Linearized flow along a geodesic of type $B$}

Among geodesics of type $B$, only those directed by $e_3$ will be considered. Let then $\gamma \colon t\in\r \mapsto \gamma_0 + te_3 \in \sol$ be such a geodesic, where $\gamma_0\in\sol$. 
The orthogonal plane to $\gamma'(t)$ in $T\sol$ is generated by $X$ and $Y$, therefore
$$
\xi_{\mathbf{|}\left(\gamma(t),\gamma'(t)\right)} = \langle X,\dot X, Y, \dot Y\rangle\;.
$$
Let
$$
g_1:=\frac12X + \dot X \quad,\quad g_2:= X \quad,\quad g_3:= -\frac12 Y + \dot Y \quad,\quad g_4:= Y \;.
$$

\begin{prop}\label{prop.index.B}
Let $\gamma \colon t\in\r \mapsto \gamma_0 + te_3 \in \sol$ be a geodesic of type $B$, where $\gamma_0 \in \sol$. 
The linearization of the geodesic flow of $\sol$ along $\gamma$ restricted to the contact distribution $\xi$ has the following matrix in the basis $(g_1,g_2,g_3,g_4)$:
$$
\left[
\begin{matrix}
	e^t & 0 	& 0	& 0\cr
	0 & e^{-t}		& 0					& 0\cr
	0 & 0		& e^{-t}			& 0\cr
	0 & 0		& 0	& e^t\cr
\end{matrix}
\right]\;,
$$
while the canonical symplectic form $\flat^*(\mathrm d p \wedge \mathrm d q)$ on the contact distribution $\xi$ along $\gamma$ is given by 
$$
\mathrm d g_1\wedge \mathrm d g_2 + \mathrm d g_3\wedge \mathrm d g_4\;.
$$

\end{prop}

\begin{proof}
The expression of the symplectic form follows from the formula (\ref{formula.sp}) obtained in the proof of Lemma~\ref{lem.sp}, since along $\gamma$, $\dot x = \dot y = 0$. 
The geodesic $\gamma$ is the intersection of the leaves of $\mathcal H'$ and $\mathcal H''$ containing it, which are totally geodesic. 
Hence the direct sum decomposition  $\xi = \xi' \oplus \xi''$, where $\xi'$ is the contact distribution of $S^*\mathcal H'$, and $\xi''$ is the contact distribution of $S^*\mathcal H''$. 
The geodesic flow restricted to $\xi'$ or $\xi''$ is the geodesic flow of the hyperbolic plane. The fields $e_1$ and $e_2$ are Killing, 
providing the second and fourth columns of the matrix. 
We can assume that $\gamma_0 = 0$. Geodesics of type $B$ of $\mathcal H'$ passing through $0\in\sol$ at $t = 0$ write
$$
\gamma_a(t) = a\frac{\sinh(t)}{\cosh(t) - c_0\sinh(t)}e_1 - \ln\left(\cosh(t) - c_0 \sinh(t)\right) e_3
$$
with $a^2 + c_0^2 = 1$, see also \cite[\S5.2]{tro}. Therefore $\sinh(t)X$ is Jacobi. Likewise, geodesics of type $B$ of $\mathcal H''$ passing through $0\in\sol$ at $t = 0$ write
$$
\gamma_b(t) = b\frac{\sinh(t)}{\cosh(t) + c_0\sinh(t)}e_2 + \ln\left(\cosh(t) + c_0 \sinh(t)\right) e_3
$$
with $b^2 + c_0^2 = 1$,  so that $\sinh(t)Y$ is Jacobi. Hence the result.
\end{proof}

\subsection{Closed $\sol$-manifolds}

\subsubsection{Classification}

Recall the following:
\begin{lem}\label{lem.homol}
Let $L$ be the suspension of a diffeomorphism of the torus ${\r^2}/{\z^2}$ defined by a linear map $A \in \glz$. Assume that $(A - I)$ is invertible too. Then, the homology with integer coefficients of $L$ satisfy the following isomorphisms 
\begin{align*}
&H_0(L;\z)\cong \z \quad;\quad
H_1(L;\z)\cong \z \oplus \left(\z^2/(A - I)(\z^2)\right) \;; \\
&H_2(L;\z)\cong \begin{cases}
\z &\textrm{ if } \det(A)>0\\
\z/2\z &\textrm{ if } \det(A)<0
\end{cases}
\quad\textrm{and}\quad
H_3(L;\z)\cong\begin{cases}
\z &\textrm{ if } \det(A)>0\\
0 &\textrm{ if } \det(A)<0
\end{cases}\;.
\end{align*}
\qed
\end{lem}


Note besides that in the situation of Lemma~\ref{lem.homol}, if $\Lambda$ is the fundamental group of $L$ based at some point $x_0 \in L$ and $\Lambda_0 \cong \z^2$ is the fundamental group of the fiber of $L \to \r/\z$ containing $x_0$, then  the exact sequence $0\to \Lambda_0 \to \Lambda \to \z \to 0$ splits. Therefore the derived subgroup  $[\Lambda, \Lambda]$ coincides with $\left( A - I\right)(\Lambda_0)$.

From Hurewicz's isomorphism, we deduce the relation
$$
\tors H_1(L;\z) \cong \Lambda_0/\left( A - I\right)(\Lambda_0)\;.
$$ 

\begin{dfn}
A linear map $A \in \glz$ is called hyperbolic iff it has two real eigenvalues different from $\pm1$.
\end{dfn}

\begin{lem}\label{lem.quotient}
Let $L$ be the suspension of a diffeomorphism of the torus ${\r^2}/{\z^2}$ defined by a hyperbolic  linear map $A \in \glz$. There exists a lattice $\Lambda$ of $\isom(\sol)$ such that $L$ is diffeomorphic to the quotient $\raisebox{-.65ex}{\ensuremath{\Lambda}}\!\backslash \sol$. 
Moreover, $\Lambda$ is generated by a lattice $\Lambda_0$ of $K$ and an isometry 
$$
l \colon (x,y,z) \in \sol \longmapsto 
\left(\varepsilon_1e^\lambda x,\varepsilon_2e^{-\lambda} y,z + \lambda \right)\in \sol
$$
where $\lambda\in\r^*$, $\varepsilon_1,\varepsilon_2\in \{\pm1\}$.
\end{lem}

\begin{proof}
We identify $\r^2$ with the derived subgroup $K$ of $\sol$ in a way that $e_1,e_2$ is a basis of eigenvectors of $A$ associated to the eigenvalues $\varepsilon_1e^\lambda$ and $\varepsilon_2e^{-\lambda}$ where $\lambda\in\r^*$, $\varepsilon_1,\varepsilon_2\in \{\pm1\}$. 
The subgroup $\z^2$ is then identified with a lattice $\Lambda_0\subset K$ invariant by $A$. Let $l$ be the product of the left multiplication by $\lambda e_3$ with the isometry $(x,y,z) \in \sol \mapsto 
\left(\varepsilon_1 x,\varepsilon_2 y,z  \right)\in \sol$.

Denote by $\Lambda$ the subgroup of $\isom(\sol)$ generated by $l$ and the left translations by elements of $\Lambda_0$, this is a lattice of $\isom(\sol)$ which satisfies the split exact sequence $0\to \Lambda_0 \to \Lambda \to \z \to 0$, where the action of $l$ by conjugation on $\Lambda_0$ coincides with the action of $A$. 
The quotient $\raisebox{-.65ex}{\ensuremath{\Lambda}}\!\backslash \sol$ is diffeomorphic to~$L$.
\end{proof}

Let $L$ be the suspension of a diffeomorphism of the torus ${\r^2}/{\z^2}$ defined by a hyperbolic  linear map $A \in \glz$. We provide $L := \raisebox{-.65ex}{\ensuremath{\Lambda}}\!\backslash \sol$ with the metric $\sol$ given by Lemma~\ref{lem.quotient}. 
The basis $B \cong S^1$ of the fibration $L \to B$ is then endowed with a metric induced by the one of $L$. 
The morphism $P$ induces a morphism $P_L \colon \isom(L) \to \isom(B)$ between their respective isometry groups.

Note that the involution $\rho^2$ induces an isometry of $L$ which belongs to the kernel of $P_L$. Likewise, a translation $(x,y,z) \in \sol \longmapsto 
\left( x + \alpha, y + \beta,z \right)\in \sol$ induces an isometry of $L$ if and only if $(\alpha,\beta) \in (A - I)^{-1}(\Lambda_0)$. We denote by $F := (A - I)^{-1}(\Lambda_0)/\Lambda_0$ this group of  translations.

\begin{lem}\label{lem.solvar}
Let $L$ be the suspension of a hyperbolic diffeomorphism of the torus endowed with its metric $\sol$ given by Lemma~\ref{lem.quotient}. 
Then, the kernel of the morphism $P_L$ is generated by $\rho^2$ and $F$ while its image is finite. The latter is reduced to isometries which preserve the orientation of $B$ when $L$ is nonorientable.
\end{lem}

\begin{proof}
The group  $\isom(L)$ coincides with the quotient by $\Lambda$ of the normalizer of $\Lambda$ in $\isom(\sol)$. An element of the kernel of $P_L$ preserves all the leaves of $\mathcal{F}$. 
It cannot induce any reflection on those leaves since the axes of these reflections would be directed by $e_1$ or $e_2$, but $\Lambda_0$ does not contain any nontrivial multiple of these elements. 
It follows that, up to multiplication by $\rho^2$, it is a translation in the fibers and then, an element of $F$. The image of $P_L$ is a subgroup of $\isom(B)$ which cannot be dense since the action on $K$ by conjugation by an element of $\sol$ close to $K$ is a linear map close to the identity, which cannot preserve $\Lambda_0$. 
Indeed, the fibers of $L$ close to a given fiber are not isometric to it. 
Thus, the image of $P_L$ is a finite subgroup of $\isom(B)$. If $\tilde k$ is such an isometry which reverses the orientation of $B$, it has a lift  $k$ of the form
$$
k(x,y,z) = \left(\eta_1e^\theta y + \alpha,\eta_2e^{-\theta} x + \beta, \theta - z \right)
$$
with $\eta_1,\eta_2\in \{\pm1\}$, $\theta,\alpha,\beta\in \r$. We get
$$
lklk^{-1}(x,y,z) = \left(\varepsilon_1\varepsilon_2(x - \alpha) + \varepsilon_1e^\lambda \alpha,
\varepsilon_1\varepsilon_2(y - \beta) + \varepsilon_2e^{-\lambda} \beta,
z \right)
$$
therefore such an isometry does not belong to $\Lambda_0$ when the sign $\varepsilon_1\varepsilon_2$ of the determinant of $A$ is negative. Hence the result.
  \end{proof}

Let $L$ be a $\sol$ variety given by Lemma~\ref{lem.solvar} and $\langle k \rangle$ be a cyclic group of isometries of $L$ acting without fixed point. If $P_L(k)$ is an isometry of the base $B$ which preserves the orientation, then the quotient of $L$ by $\langle k \rangle$ is also the suspension of a hyperbolic diffeomorphism of the torus. 
Should  the opposite occur,  $P_L(k)$ is a reflection of $B$ and we can assume that $k$ is of order $2$. The quotient $L/\langle k \rangle$ is no longer a bundle over $B$ and is orientable. 
Indeed, $L$ is necessarily orientable from Lemma~\ref{lem.solvar}, while over a fixed point of $P_L(k)$, the linear map associated to $k$ cannot be a rotation by an  angle of $\frac\pi2 \mod \pi$, it must be then a reflection  in the associated fibers, therefore $k$ preserves the orientation of~$L$.

\begin{dfn}
Following \cite{morimoto}, we call sapphire the quotient of a $\sol$-bundle $L$ given by Lemma~\ref{lem.solvar} by an involutive isometry acting without fixed point and inducing a reflection on the basis $B$.
\end{dfn}

The second homology group with integer coefficients of a sapphire vanishes, its first homology group is torsion. We call $\sol$-manifold any manifold obtained as a quotient of $\sol$ by a discrete subgroup of isometries acting without fixed point. Recall the

\begin{thm}\label{thm.classif}
The closed $\sol$-manifolds are the sapphires and the suspensions of diffeomorphisms of the torus ${\r^2}/{\z^2}$ defined by hyperbolic  linear maps.
\end{thm}

\begin{proof}
By definition, sapphires are closed $\sol$-manifolds while suspensions of hyperbolic diffeomorphisms of the torus are $\sol$ by Lemma~\ref{lem.quotient}.

Conversely, let $\Lambda \subset \isom(\sol)$ be a cocompact discrete subgroup acting wihout fixed point on $\sol$. Let $\Lambda_0$ be the kernel of the restriction of $P$ to $\Lambda$. 
An element of $\Lambda_0$ writes $gh$ where $g$ is a translation of vector $\alpha e_1 + \beta e_2 \in \sol$, $\alpha,\beta \in \r$ and $h \in \left\{id,r_X,r_Y\right\}$, since $\Lambda_0$ acts without fixed point and preserves all the leaves of $\mathcal{F}$. 
The subgroup of translations of $\Lambda_0$ is of index at most $2$ in $\Lambda_0$ and is necessarily of rank $2$, see for example~\cite[Theorem~4.17]{scott}.

Let $id \ne l\in \Lambda$ be such that $P(l)$ preserves the orientation of $B$. 
The quotient of $\sol$ by the subgroup generated by $l$ and the translations of $\Lambda_0$ is a torus bundle over the circle with hyperbolic monodromy. 
The result  then follows from Lemma~\ref{lem.solvar}.
\end{proof}

\subsubsection{Closed geodesics}\label{subsec.closedgeodes}

Let $L = \raisebox{-.65ex}{\ensuremath{\Lambda}}\!\backslash \sol$ be a closed $\sol$-manifold given by Theorem~\ref{thm.classif}. 
The lattice $\Lambda$ satisfies the exact sequence  $0\to \Lambda_0 \longrightarrow \Lambda \stackrel{P_L}{\longrightarrow} \Lambda/\!\raisebox{-.65ex}{\ensuremath{\Lambda_0}} \to 0$, where $\Lambda_0 \subset K$. 
We denote by $p \colon L \longrightarrow B$ the associated map, where $B = \r/\!\raisebox{-.65ex}{\ensuremath{P_L(\Lambda)}}$ is homeomorphic to an interval if $L$ is a sapphire and to the circle otherwise. 
Any periodic geodesic $\gamma \colon \r \longrightarrow L$ has a lift which is a geodesic $\widetilde\gamma \colon \r\longrightarrow \sol$. 
We will say that $\gamma$ is of type $A$, $B$, or $C$ if $\widetilde\gamma$ is of type $A$, $B$, or $C$ in the sense of \cite{tro}. 
Closed geodesics of type $A$ of $L$ are in particular quotients of  geodesics of type $A$ of $\sol$ directed by elements of $\Lambda_0$. 
These geodesics are contained in the fibers of $p$ and then belong only to a dense countable subset of such fibers.

\begin{lem}\label{lem.closed}
Let $L$ be a closed $\sol$-manifold. Then, any closed geodesic of type $C$ of $L$ is homotopic to a closed geodesic of type $A$ of $L$. 
Furthermore, closed geodesics of type $B$ of $L$ are quotients of geodesics of type $B$ of $\sol$ directed by $e_3$, that is 
intersection of hyperbolic leaves of $\mathcal{H'}$ and $\mathcal{H''}$ in $\sol$. 
\end{lem}
 
\begin{proof}
Let $\gamma \colon \r \longrightarrow L$ be a periodic geodesic of type $C$ and let 
$\widetilde\gamma \colon t\in \r\longrightarrow \left(\tilde x(t),\tilde y(t),\tilde z(t)\right) \in\sol$ be a lift of $\gamma$.
There exists $l_0 \in \Lambda_0$ such that for every $t\in\r$, $\widetilde \gamma (t + T) = l_0\cdot \widetilde \gamma(t)$, where
$T$ is the minimal period of $\gamma$. 
In particular, the coordinate $\tilde z$ of $\widetilde \gamma$ is $T$-periodic, by \cite[\S4.4]{tro}. This forces $T$ to be a multiple of 
$\frac {4K}{\mu}$, where $K$ and $\mu$ are the quantities introduced in \cite{tro}. 
We deduce from the equation of geodesics of type $C$ obtained in \cite[\S4.4]{tro} the relation 
$$
\widetilde \gamma(t + T) - \widetilde \gamma(t) = 2\left(L\mu\sqrt{\vert ab\vert}\right)T \left(\pm X\pm Y\right)\;.
$$
Hence, writing $T = n \left(\frac {4K}{\mu}\right)$, $n\in\n^*$, we deduce that
$$
\left(8nLK\sqrt{\vert ab\vert}\right)\left(\pm X\pm Y\right) = l_0\;,
$$
so that the closed geodesic $\gamma$ is homotopic to the closed geodesic of type $A$ of $L$ defined by $l_0$. The latter's length is a multiple of the quantity $8LK\sqrt{\vert ab\vert}$, with the notations of \cite{tro}. Likewise, let $\gamma \colon \r \longrightarrow L$ be a periodic geodesic of type $B$, of minimal period $T$, and let 
$$
\widetilde\gamma \colon t\in \r\longrightarrow \left(\widetilde\gamma_h(t),\widetilde\gamma_v(t)\right) \in\sol/\!\raisebox{-.65ex}{\ensuremath{\Lambda_0}} = \left(K/\!\raisebox{-.65ex}{\ensuremath{\Lambda_0}}   \rtimes \r \right)
$$ 
be a lift of $\gamma$ to the infinite cyclic covering of $L$. There exists $l \in \Lambda/\!\raisebox{-.65ex}{\ensuremath{\Lambda_0}}$ of infinite order such that for all $t\in\r$,  $\widetilde \gamma (t + T) = l\cdot \widetilde \gamma(t)$. The action of $l$ on the torus $K/\!\raisebox{-.65ex}{\ensuremath{\Lambda_0}}$ is defined by a hyperbolic linear map $A$. We deduce that for all $t\in\r$, $\left(A - I\right)\left(\widetilde\gamma_h(t)\right) = 0$. Hence, $\widetilde\gamma_h$ is necessarily constant and equal to a fixed point of $A$.
\end{proof}

\begin{rem}\label{rem.estimate}
The proof of Lemma~\ref{lem.closed} provides an estimate of the length of closed geodesics of type $A$ homotopic to closed geodesics of type $C$. This estimate will be crucial in the proof of Proposition~\ref{prop.metric.choice}. 
Likewise, if $L$ is the suspension of a diffeomorphism of the torus defined by a hyperbolic  linear map $A \in \glz$, we deduce that closed geodesics of type $B$ of $L$ are in correspondence with the periodic points of $A \colon {\r^2}/{\z^2} \longrightarrow {\r^2}/{\z^2}$.
\end{rem}

\begin{prop}\label{prop.metric.choice}
Let $L$ be a closed three-dimensional manifold given by Theorem~\ref{thm.classif} and let $\Pi$ be a finite subset of homotopy classes of $L$. 
There exists a $\sol$-metric on $L$ such that no element of $\Pi$ gets realized by a closed geodesic of type $C$ of $L$. 
Furthermore, this metric can be chosen such that closed geodesics of type $A$ of $L$ homotopic to elements of $\Pi$ are of Morse-Bott index~$1$.
\end{prop}

\begin{proof}
From Theorem~\ref{thm.classif}, the $\sol$-manifold $L$ is diffeomorphic to the quotient of $\sol$ by a lattice $\Lambda \subset \isom(\sol)$ satisfying the exact sequence $0\to \Lambda_0 \longrightarrow \Lambda \stackrel{P_L}{\longrightarrow} \Lambda/\!\raisebox{-.65ex}{\ensuremath{\Lambda_0}} \to 0$ where $\Lambda_0 \subset K$ is a lattice. 
The fundamental group of $L$ is therefore isomorphic to $\Lambda$, and from Lemma~\ref{lem.closed}, only classes in $\Pi\cap \Lambda_0$ can be realized by closed geodesics of type $A$ or $C$. 
Up to multiplication of the lattice $\Lambda_0$ by  a constant $0<\varepsilon \ll 1$, we can assume that all the elements of $\Pi\cap \Lambda_0$ have length bounded from above by $4 - \pi$. Such a $\sol$-metric fits. 
Indeed, from Lemma~\ref{lem.closed} and Remark~\ref{rem.estimate}, every closed geodesic of type $C$ of $L$ is homotopic to a closed geodesic of type $A$ of length a multiple of $8LK\sqrt{\vert ab\vert}$, adopting the notations of \cite{tro}. 
Now, taking again the notations of \cite{tro}, we get
$$
8LK\sqrt{\vert ab\vert} = \frac{8}{\sqrt 2 \sqrt{1 + k^2}}\left(E - \frac K2(1 - k^2)\right)\textrm{ where $0\leq k \leq 1$.}
$$
Moreover, $E = \displaystyle\int^\frac{\pi}2_0\sqrt{1 - k^2\sin^2\theta}\,\mathrm d \theta \geq 1$ and $$
K\sqrt{1 - k^2} = \displaystyle\int^\frac{\pi}2_0\sqrt{\frac{1 - k^2}{1 - k^2\sin^2\theta}}\,\mathrm d \theta \leq \frac\pi2\;.
$$
We get the estimate $8LK\sqrt{\vert ab\vert} \geq 4 - \pi$ which prevents the geodesic of type $C$ to be homotopic to an element of $\Pi$. 
Likewise, the length of closed geodesics of type $A$ homotopic to elements of $\Pi$ are less than $4 - \pi < \frac{2\pi}{\sqrt{2}}$. From Proposition~\ref{prop.index.A}, the Conley-Zehnder index of these geodesics in the trivialisation $(h_1,\dots,h_4)$ of $\xi$ is $1$. 
Indeed, the Conley-Zehnder index of the rotation block is $1$ by definition, while the (Bott-)Conley-Zehnder index of the unipotent block 
$
U = 
\left[
\begin{matrix}
	1 & t 	\cr
	0 & 1 \cr
\end{matrix}
\right]
$
vanishes, see the thesis of F.~Bourgeois. Indeed, this block is solution of the differential equation $\dot U = S\mathcal{J}U$ with $U(0) = I$, 
$S = 
\left[
\begin{matrix}
	1 & 0 	\cr
	0 & 0 \cr
\end{matrix}
\right]$
and
$\mathcal{J} = 
\left[
\begin{matrix}
	0 & 1 	\cr
	-1 & 0 \cr
\end{matrix}
\right]$. By definition, the (Bott-)Conley-Zehnder index of this block is the Conley-Zehnder index of the solution of the differential equation  $\dot V = \left( S - \delta I\right)\mathcal{J} V$ with $V(0) = I$ and $0<\delta \ll 1$, which is hyperbolic. The result follows from \cite[Theorem~3.1]{V}, \cite[Proposition~1.7.3]{EGH} which identifies this Conley-Zehnder index to the Morse-Bott index.
\end{proof}

The $\sol$-metrics given by Proposition~\ref{prop.metric.choice} are metrics for which the area of the fibers of the map $p \colon L \longrightarrow B$ is not too large compared to the length of $B$. 
In fact, without changing the length of $B$, it is possible to expand or contract the fibers of $p$ as much as we want, keeping the $\sol$ feature of the metric. This observation was crucial in the proof of Proposition~\ref{prop.metric.choice} and will be very useful in Section~2.

\section{$\sol$ Lagrangian submanifolds in uniruled symplectic manifolds}

\subsection{Statement of the results}

\begin{dfn}\label{dfn.uniruled}
We say that a closed symplectic manifold $(X,\omega)$ is uniruled iff it has a non vanishing
genus $0$  Gromov-Witten invariant of  the form
$
\langle [pt]_k;[pt],\omega^k \rangle^X_A\;,
$
where $A \in H_2(X;\z)$, $k\geq 2$, and $[pt]_k$ represents the Poincar\'e dual of the point class in the moduli space $\overline{\mathcal{M}}_{0,k+1}$ of  genus $0$ stable curves with $k+1$ marked points.
\end{dfn}

This Definition~\ref{dfn.uniruled} differs from \cite[Definition~4.5]{hu-li-ruan} where $\omega^k$ is replaced by any finite set of differential forms on $X$. Nevertheless, from \cite[Theorem~4.2.10]{kollar.sp.uni}, complex projective uniruled manifolds are all symplectically uniruled in the sense of Definition~\ref{dfn.uniruled}. The advantage for us to restrict ourselves to Definition~\ref{dfn.uniruled}
is that for every Lagrangian submanifold $L$ of $X$, the form $\omega$ has a Poincar\'e dual representative disjoint from $L$. 

Our goal is to prove the following results.

\begin{thm}\label{thm.main}
Let $(X,\omega)$ be a closed uniruled  symplectic manifold of dimension six. For any Lagrangian submanifold $L$ of $X$ homeomorphic to the suspension of a hyperbolic diffeomorphism of the two-dimensional torus, there exists a symplectic disc of Maslov index zero with boundary on $L$. Furthermore, such a disc can be chosen such that its boundary does not vanish in $H_1(L;\q)$.
\end{thm}

In particular, such a Lagrangian submanifold $L \hookrightarrow X$ given by Theorem~\ref{thm.main} cannot be monotone.  It might be true that such Lagrangian submanifolds do not exist at all, see 
\S \ref{rems}. In fact, in the case of the projective space, the absence of orientable Sol Lagrangian submanifolds follows from Theorem 14.1 of \cite{Fuk}. 
Moreover, in this paper Kenji Fukaya remarks that his methods may extend to uniruled manifolds as well. 
Nevertheless, we deduce the following corollaries.

\begin{cor}\label{cor.dp}
Let $p \colon (X,c_X) \to (B,c_B)$ be a  dominant real morphism with rational fibers,  where $(X,c_X)$ (respectively   $(B,c_B)$) is a real algebraic manifold of dimension $3$ (respectively $1$). Then, the real locus of $X$ has no $\sol$ component $L \subset X^{nonsing}$ such that the restriction of $p$ to $L= \raisebox{-.65ex}{\ensuremath{\Lambda}}\!\backslash \sol$ is the map $L \to  \r/\!\raisebox{-.65ex}{\ensuremath{P_L(\Lambda)}}$ defined in \S~\ref{subsec.closedgeodes}.
\end{cor}

In particular, the restriction of $p$ to $L$ is a submersion if $L$ is the suspension of a hyperbolic diffeomorphism of the torus and has two multiple fibers if $L$ is a sapphire. 
Note that Koll\'ar proved in \cite{koIV} that in the situation of Corollary~\ref{cor.dp}, an orientable $\sol$ component $L$ of $X(\r)$ automatically satisfies the last conditions. That is $L $ is contained in the
nonsingular part $X^{nonsing}$ of $X$ and the restriction of $p$ to $L= \raisebox{-.65ex}{\ensuremath{\Lambda}}\!\backslash \sol$ is the map $L \to  \r/\!\raisebox{-.65ex}{\ensuremath{P_L(\Lambda)}}$.
Corollary~\ref{cor.dp} means that in \cite[Theorems~1.1 and 1.3]{koIV}, the manifold $N$ cannot be endowed with a $\sol$ metric, confirming the expectation of Koll\'ar discussed in Remark~1.4 of this paper. The upshot is that if $X$ is a projective uniruled manifold defined over $\R$ with orientable
real locus, then, up to connected sums with $\R P^3$ or $S^2 \times S^1$ and modulo finitely many
closed three manifolds, every component of $\R X$ is a Seifert fiber space or a connected sum of
Lens spaces.

\begin{proof}[Proof of Corollary~\ref{cor.dp}]
Choosing an appropriate branched covering $ (B',c_{B'})\to (B,c_B)$ and resolving the singularities of the fibered product $X \times_p B'$, we get a nonsingular uniruled real projective variety $Y$ containing in its real locus a connected component $L'$ homeomorphic to the suspension of a hyperbolic diffeomorphism of the torus. In this construction, $B'$ can be obtained of positive genus and such that the projection $p_* \colon H_1(L';\q)\to H_1(B';\q)$ is injective. It follows that $H_1(L';\q)$ injects into $H_1(Y;\q)$ and Theorem~\ref{thm.main} provides the contradiction.
\end{proof}

\begin{cor}\label{cor.fano}
The real locus of a smooth three-dimensional Fano manifold does not contain any connected component homeomorphic to the suspension of a hyperbolic diffeomorphism of the two-dimensional torus. \qed
\end{cor}

Indeed, in the situation of Corollary~\ref{cor.fano}, the real locus would be monotone. Finally, we deduce.

\begin{cor}\label{cor.torus}
A Hamiltonian diffeomorphism of a uniruled symplectic four-manifold which preserves some Lagrangian torus $T$ cannot restrict to
a hyperbolic diffeomorphism of $T$. 
\end{cor}

Note that in the case of weakly exact Lagrangian submanifolds, a stronger result has recently been obtained by Shengda Hu and
Fran{\cedilla{c}}ois Lalonde in \cite{HuLal}. 

\begin{proof}[Proof of Corollary~\ref{cor.torus}]
If such a Hamiltonian diffeomorphism $\phi$ of a uniruled symplectic four-manifold $X$ would exist, there would exist a path
$\phi_t$, $t \in [ 0,1]$, of Hamiltonian diffeomorphisms of $X$ such that $\phi_t$ equals the identity (resp. $\phi$) for $t$ close
to zero (resp. one). Let us denote by $Y$ the product of $X$ with a genus one curve $B$ and equip this manifold
with the symplectic form $\omega_Y$ obtained as a sum of the pulled back of  $\omega$ with the pulled back of a volume form 
$\omega_B$ on $B$. This symplectic six-manifold $Y$ is also uniruled. 
Indeed, we can equip it with a product almost complex structure for which the projection onto $B$ is a 
J-holomorphic map. 
Fixing a point $x$ in some fiber, all rational J-holomorphic curves passing through $x$ are contained in this fiber. 
Since the index of rational curves in $Y$ is just one plus the index of rational curves in $X$, we deduce that
the Koll\'ar's Gromov-Witten invariant of $X$ equals the Koll\'ar's Gromov-Witten invariant of $Y$ in the corresponding class. 
Since this invariant does not vanish, $Y$ is uniruled in the sense of Definition~\ref{dfn.uniruled}. 

Now, let $U$ be a meridian of $B$ and $L \subset Y$ be the three-manifold defined as the union for $t \in U$
of $L_t = \phi_t (T)$ in $Y_t = X$. Here, we denote by $Y_t$ (resp. $L_t$) the fiber of $Y$ (resp. $L$) over $t \in U$
and identify $U$ with the interval $[ 0,1]$ with glued ends. The manifold $L$ thus defined is diffeomorphic to the suspension
of $\phi$ but is not yet Lagrangian. The restriction of $\omega_Y$ to $L$ coincides with  the pulled back of $\omega$.
The latter on $L$ equals $dH \wedge dt$, where $dt$ is the pulled back volume form of $U$ and $H$ is the time dependent
Hamiltonian function defining $\phi_t$, $t \in [ 0,1]$ (compare \cite{AkSal}). The difference between $\omega_Y$ and the 
globally defined $dH \wedge dt$ gives a closed two-form on $Y$ in the same cohomology class as $\omega_Y$ and
which still coincides with $\omega$ on every fiber of $Y \to B$. From a Theorem of Thurston, this new form becomes symplectic
after adding a big multiple of the pulled back of $\omega_B$, see  \S $6.1$ of \cite{McDSal}. For the latter, $L$ remains
Lagrangian and $Y$ uniruled since it is deformation equivalent to $\omega_Y$. Theorem~\ref{thm.main} now
provides the contradiction.
\end{proof}

The proof of Theorem~\ref{thm.main} uses symplectic field theory and thus is inspired by the proof of \cite[Th.1.7.5]{EGH} (see also \cite{sftVit} and \cite{KhBourbaki}). The strategy is the following: let $A\in H_2(X;\z)$ and $k \geq 2$ be given by Definition~\ref{dfn.uniruled}. We choose $k$ submanifolds $H_1,\dots,H_k$ of codimension $2$ in $X$,  pairwise transversal,  disjoint from $L$, and Poincar\'e dual to $\omega$. We choose also some points $x\in L$ and $p_k \in \mathcal{M}_{0,k+1} \subset \overline{\mathcal{M}}_{0,k+1}$. From a theorem of Weinstein \cite{weinstein} we know that $L$ possesses a neighborhood $U$ bounded by a contact hypersurface $S$ isometric to the unitary cotangent bundle of $L$ for a $\sol$-metric given by Lemma~\ref{lem.quotient}. Let $J$ be a generic almost-complex structure on $X$ singular along $S$ given by symplectic field theory, see \S~\ref{subsec.singular}. From the compactness Theorem  \cite{BEHWZ}, the rational $J$-holomorphic curves counted by the invariant $\langle pt_k;x,H_1,\dots,H_k \rangle^X_A$ are punctured nodal curves, each irreducible component of which is properly embedded either in $U$, or in $X\setminus U$. Moreover, at their punctures, these curves converge to closed Reeb orbits of $S$, which correspond to closed geodesics of type $A$, $B$, or $C$ of $L$, see \S~\ref{subsec.closedgeodes}.

The first step of the proof consists of showing that $S$ and $J$ can be chosen in such a way that only geodesics of type $A$ of Morse-Bott index $1$ and geodesics of type $B$ may appear as limits of such components at their punctures. The manifold $S$ with this property  corresponds to a $\sol$-metric of $L$ for which the fibers have small volume compared to the length of the basis $B$ of $L$. The structure $J$ is singular along a finite number of such hypersurfaces $S_1,\dots,S_N$ for which the volume of the fibers decreases with respect to the length of the base, or close to such a singular structure. In other words, we decompose $(X,\omega)$ into a symplectic cobordism whose pieces are $U=U_0$, $X\setminus U_N$, and $U_i\setminus U_{i-1}$ for $1 \leq i \leq N$, where $U_i$ are the Weinstein neighborhoods of $L$ with boundary $S_i$. 

      
\begin{figure}[ht]
\centering
\includegraphics[height =3.2cm]{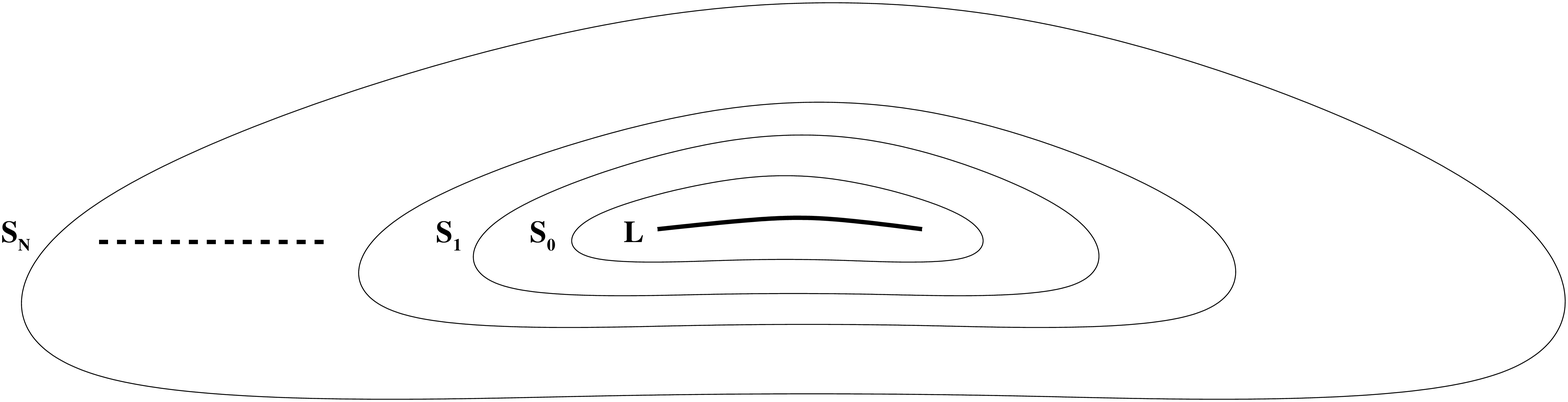}
\caption{Weinstein neighborhoods.}
	\label{fig.neighborhoods}
\end{figure}

We show furthermore, again assuming that $X$ does not possess any symplectic disc of 
vanishing Maslov index and boundary on $L$ nontrivial in $H_1(L,\q)$, that all the components of these rational curves which are in $U=U_0$ are $J$-holomorphic cylinders and that only one of them is asymptotic to a geodesic of type $A$. These rational curves are all broken into a finite union of cylinders closed by two planes as in Figure~\ref{fig.cylinders}.


\begin{figure}[ht]
\centering
\includegraphics[height =.9cm]{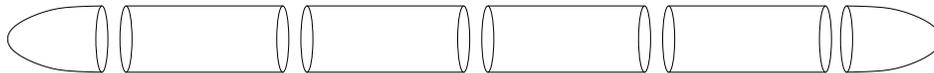}
\caption{Broken curve.}
        \label{fig.cylinders}
\end{figure}

The second step of the proof consists of showing that the degree of the evaluation map $eval_0 \colon\mathcal{M}_{0,k+1}^{A,p_k}(H,L;J)  \to L$ vanishes.
Indeed, each cylinder of $U$ asymptotic to a geodesic of type $A$ lifts canonically to the infinite cyclic covering $\widetilde U$ of $U$  once chosen a lift of this closed geodesic. 
Hence, the whole compact family $\mathcal{M}$ of cylinders of $U$ asymptotic to the type $A$ geodesics and touching $L$ lifts to a compact family of cylinders in this covering $\widetilde U$.
As a consequence, the evaluation map  $\mathcal{M} \to L$ decomposes through the infinite cyclic covering $\widetilde L$ of $L$ as $\mathcal{M} \to \widetilde L \to L$. This forces its degree to vanish. But the later equals the Gromov-Witten invariant $\langle [pt]_k;[pt],\omega^k \rangle^X_A$ which is nontrivial by assumption, hence the contradiction. 
 
\subsection{Singular almost complex structures and stable curves}\label{subsec.singular}

\subsubsection{Singular almost complex structures}

We recall the following definitions  coming from symplectic field theory \cite{EGH} 
(compare \cite[\S2.1]{welsch}). In this paragraph, $(X , \omega)$  may be any $2n$-dimensional symplectic manifold. 

\begin{dfn}
A $S$-neck of the manifold $(X , \omega)$  is an embedding $\phi : S \times [- \epsilon , \epsilon ] \to X$
which satisfies $\phi^* \omega = d(e^t \theta)$, where $(S , \theta)$ is a closed contact manifold of dimension $2n-1$, $\epsilon \in \R_+^*$ and $t \in [- \epsilon , \epsilon ]$.
\end{dfn}

\begin{dfn}
\label{defssing}
An almost-complex structure $J$ of $X$ is called $S$-singular if there exists a $S$-neck $\phi : S \times [- \epsilon , \epsilon ] \to X$
such that:

1) The domain of definition of $J$ is the complement $X \setminus \phi (S \times \{ 0 \})$.

2) The almost-complex structure $\phi^* J$ preserves the contact distribution $\ker (\theta)$
of $S \times \{ t \}$ for every $t \in  [- \epsilon , \epsilon ]   \setminus \{ 0 \}$ and its restriction to
$\ker (\theta)$ does not depend on $t \in  [- \epsilon , \epsilon ]   \setminus \{ 0 \}$.

3) $\forall (x,t) \in S \times ([- \epsilon , \epsilon ]   \setminus \{ 0 \})$, $\phi^* J (\frac{\partial}{\partial t}) \vert_{(x,t)} =
\alpha' (t) R_\theta \vert_{(x,t)}$, where $\alpha' : [- \epsilon , \epsilon ]   \setminus \{ 0 \} \to \R_+^*$ is even with infinite
integral and $R_\theta$ denotes the Reeb vector field of $(S , \theta)$.
\end{dfn}

\begin{dfn}
An almost-complex structure $J$ of $(X , \omega)$ is called singular if it is $S$-singular for some
$(2n-1)$-dimensional contact manifold $(S , \theta)$.
\end{dfn}

Denote by $\partial {\cal J}_\omega$ the space of singular almost-complex structures of $X$ compatible
with $\omega$. It is equipped with the following topology. A singular almost-complex structure $J$
is said to be in the $\eta$-neighborhood of $J_0 \in \partial {\cal J}_\omega$, $\eta > 0$, if these structures
are $S$-singular for the same contact manifold $(S , \theta)$ and if there exist pairs $(\phi_0 , \alpha'_0)$
and $(\phi , \alpha')$ given by Definition \ref{defssing} such that:

1) The distance between $\phi$ and $\phi_0$ is less than $\eta$. This distance in the space of embeddings
of finite regularity is induced by some fixed metric on $X$. The regularity of these embeddings is one more
than the regularity of the almost-complex structures which throughout the paper is supposed to be finite.

2) There exists $0 < \delta < \epsilon$ such that $2\eta \int_\delta^\epsilon \alpha'_0 (t) dt > 1$ and
the distance between the restrictions of $J$ and $J_0$ to the complement $X \setminus \phi_0 (S \times ]- \delta , \delta [)$
is less than $\eta$.

\begin{dfn}
\label{defjsneck}
An almost-complex structure $J \in {\cal J}_\omega$ is said to have an $S$-neck if $X$
has an $S$-neck $\phi : S \times [- \epsilon , \epsilon ] \to X$ such that

1) The almost-complex structure $\phi^* J$ preserves the contact distribution $\ker (\theta)$
of $S \times \{ t \}$ for every $t \in  [- \epsilon , \epsilon ]$ and its restriction to
$\ker (\theta)$ does not depend on $t \in  [- \epsilon , \epsilon ] $.

2) $\forall (x,t) \in S \times [- \epsilon , \epsilon ]$, $\phi^* J (\frac{\partial}{\partial t}) \vert_{(x,t)} =
\alpha' (t) R_\theta \vert_{(x,t)}$, where $\alpha' : [- \epsilon , \epsilon ]   \to \R_+^*$ is even.

The integral $\int_{- \epsilon}^\epsilon \alpha' (t) dt  $ is called the length of the neck.
\end{dfn}

Hence, an $S$-singular almost-complex structure is an almost-complex structure having an $S$-neck
of infinite length. This terminology comes from symplectic field theory \cite{EGH}. Indeed, if $J \in {\cal J}_\omega$
has an $S$-neck and $\alpha$ is the odd primitive of the function $\alpha'$ given by Definition \ref{defjsneck},
then, the diffeomorphism $(x,t) \in S \times [- \epsilon , \epsilon ] \mapsto (x, \alpha (t)) \in S \times [\alpha (-\epsilon) , \alpha (\epsilon) ]$
pushes forward $J$ to an almost-complex structure which preserves the contact distribution and sends the Liouville
vector field $\frac{\partial}{\partial t}$ onto the Reeb vector field $R_\theta$, compare \S $2.2$ of \cite{HWZfinite}.
In the language of symplectic field theory, a symplectic manifold $(X , \omega)$ equipped with a
$S$-singular almost-complex structure $J$ is an almost-complex manifold $(X \setminus \phi (S \times \{ 0 \}) , J)$
with cylindrical end.

Set $\overline{\cal J}_\omega = {\cal J}_\omega \sqcup \partial {\cal J}_\omega$ and equip this space with the following topology.
An almost-complex structure $J \in {\cal J}_\omega$
is said to be in the $\eta$-neighborhood of the $S$-singular almost-complex structure $J_0 \in \partial {\cal J}_\omega$, $\eta > 0$,
if it has an $S$-neck and there exists pairs $(\phi_0 , \alpha'_0)$,
$(\phi , \alpha')$ given by Definitions \ref{defssing} and \ref{defjsneck} such that:

1) The distance between $\phi$ and $\phi_0$ is less than $\eta$ in the space of embeddings
of our fixed finite regularity.

2) There exists $0 < \delta < \epsilon$ such that $2\eta \int_\delta^\epsilon \alpha'_0 (t) dt > 1$ and
the distance between the restrictions of $J$ and $J_0$ to the complement $X \setminus \phi_0 (S \times ]- \delta , \delta [)$
is less than $\eta$.

In particular, when $\eta$ is closed to zero, the length of the $S$-neck of $J$ is closed to infinity.

\subsubsection{Stable curves}

We recall that the combinatorial type of a punctured nodal curve of arithmetical genus $0$ is encoded by a tree, see \cite[Definition~6.6.1]{KontManin}. The vertices of this tree correspond to the irreducible components of the curve, the edges of valence $2$ to the nodes of the curve, and the edges of valence $1$ to the punctures. 

Such a curve is called stable whenever each vertex bounds at least three edges of the tree. A special point of the curve is a puncture or a node of the curve, these points are in one-to-one correspondence with the edges of the associated tree. 

\begin{dfn}
A special point of an irreducible component $D$ of a punctured nodal curve of arithmetical genus $0$ is called essential if and only if it is either a puncture or a node such that the attached curve has at least one puncture. 

Such a component is called essential iff it contains at least three essential points.
\end{dfn}

If $C$ is a nodal curve  of arithmetical genus zero with at least three punctures, the associated stable curve is obtained by contracting all the non-essential irreducible components of $C$.
We denote by $\overline{\mathcal{M}}_{0,k}$ the moduli space of genus zero stable curves with $k$ punctures, see~\cite{Keel}. 

\begin{dfn}
A punctured nodal curve of genus $0$ is called string-like when all its irreducible components have at most two nodes.
\end{dfn}

Hence, the associated tree of a string-like nodal curve is of type $A_n$ after removing the edges of valence $1$ corresponding to the punctures of the curve.

\subsection{Singular structures adapted to $\sol$ Lagrangian submanifolds}

\begin{nota}\label{nota.gene}
Let us fix some notations which we will use in the sequel.

\begin{description}
\item[ ] $(X,\omega)$: closed uniruled symplectic manifold of dimension $6$.
\item[ ] $L$: Lagrangian submanifold of $(X,\omega)$ homeomorphic to the suspension of a hyperbolic diffeomorphism of the two-torus.
\item[ ] $A$: Element of $H_2(X;\z)$ given by Definition~\ref{dfn.uniruled}.
\item[ ] $k$: Integer $\geq 2$, given by Definition~\ref{dfn.uniruled}.
\item[ ] $H_1,\dots,H_k$: submanifolds of $(X,\omega)$, disjoint from $L$, transversal one to each other and Poincar\'e dual to $\omega$.
\item[ ] $x$: Point of $L$.
\item[ ] $p_k$: Point of $\mathcal{M}_{0,k+1} \subset \overline{\mathcal{M}}_{0,k+1}$.
\item[ ] $g_0$: $\sol$ metric on $L$ given by Lemma~\ref{lem.quotient}.
\item[ ] $U_0$: Weinstein neighborhood of $L$, disjoint from $H_1,\dots,H_k$, whose boundary is isomorphic to the unitary cotangent bundle of $(L,g_0)$. 
\item[ ] $S_0$: Boundary of $U_0$.
\item[ ] $J_0$: $\omega$-positive generic $S_0$-singular almost-complex structure on $X$.
\end{description}
\end{nota}

\begin{dfn}
We say that a nodal $J_0$-holomorphic curve $C \subset X$ of arithmetical genus $0$ passing through $x$ represents the class $p_k$ iff there exists $x_i \in C \cap H_i$, $1 \leq i \leq k$, such that the  stable curve associated to $C\setminus \{x,x_1,\dots,x_k\}$ represents $p_k \in \mathcal{M}_{0,k+1}$. We say that it represents $(A,p_k)$ if furthermore, it is homologous to $A$.
\end{dfn}

\begin{lem}\label{lem.nodal.curves}
Let $C$ be a rational $J_0$-holomorphic curve of $(X,\omega)$ representing $(A,p_k)$. Assume:
\begin{enumerate}
\item That $(X,\omega)$ contains no symplectic disc of Maslov index zero whose boundary on~$L$ 
does not vanish in $H_1(L;\q)$.
\item That all closed geodesics of $(L,g_0)$ associated to the nodes of $C$ are either of type $B$ or of type $A$ with Morse-Bott index one.
\end{enumerate}

Then, $C$ is string-like. Furthermore, the cylinder in $C\cap U_0$ containing $x$ is the only one which converges to geodesics of type $A$.
\end{lem}

\begin{proof}
We begin with the computation of the index of each irreducible component of $C$, that is the expected dimension of the moduli space containing this component. 
Let $D$ be such a component, it is isomorphic to a punctured Riemann sphere. By assumption and from \cite{HWZ1}, $D$ is asymptotic at each of its punctures to a cylinder on a geodesic of type $A$ or $B$. 
The normal bundle of $D$ at its punctures of type $A$ is trivialized by Proposition~\ref{prop.index.A}. If $L$ is the suspension of a diffeomorphism of the torus with positive eigenvalues, then the normal bundle of $D$ at its punctures of type $B$ is trivialized by Proposition~\ref{prop.index.B}. Otherwise, for instance if the eigenvalue associated to $X$ is negative, we perturb the trivialization in the $\langle g_1,g_2 \rangle$-plane by a rotation whose angle only depends of the coordinate along $e_3$, in such a way that this angle is an odd multiple of $\pi$ at the altitude $\lambda$ associated to the eigenvalues of our diffeomorphism. This trivialization in $T^*\sol$ induces on the quotient a trivialization of the normal bundle of $D$ at its punctures of type $B$. Moreover, the Conley-Zehnder index of these geodesics, calculated in this trivialization only depends on the homology class of the geodesic in $H_1(L;\z)/\tors$. Finally, our trivializations can be extended to trivializations of $TX$ along the geodesics, by adding the Liouville vector field and the  vector field tangent to the geodesics. Moreover these chosen trivializations extend to trivializations of $TU_0$. 

We denote by $\mu_{CZ}(p)$ the Conley-Zehnder index of a puncture $p$ of $D$ computed in our chosen trivialization, and by $\mu(D)$ twice the obstruction to extend this trivialization of $TX$ at punctures to the whole $D$. We just saw that if $D \subset U_0$, then $\mu(D) = 0$. From \cite{Bour} (see also \cite{HWZIII} and \cite{WelsOpt}), the index of $D$ is given by the following Riemann-Roch formula.
$$
\ind_\r(D) = 
\#\{\textrm{punctures of type $A$}\} \pm \sum_{p \in\{\textrm{punct. of $D$}\} } \mu_{CZ}(p)
+ \mu(D)\;,
$$
since the dimension of $X$ is $6$, where the $\pm$ sign depends on whether $D\subset U_0$ or $D \subset X\setminus U_0$.

When $D\subset U_0$, we deduce that $\ind_\r(D) = 
2 \#\{\textrm{punctures of type $A$}\}$, since $\mu(D) = 0$, the total homology class of geodesics of type $B$ vanishes and $\mu_{CZ}(p) = 1$ if $p$ is of type $A$ and $D \subset U_0$, see Proposition~\ref{prop.metric.choice}. In particular, this index can only increase under branched coverings, so that the moduli space containing $D$ is of the expected dimension $2 \#\{\textrm{punctures of type $A$}\}$. 

Likewise, if $D \subset X\setminus U_0$ is a branched covering of $D'$ of degree $d$, then 
\begin{eqnarray*}
\ind_\r(D) &= &
- \#\{\textrm{punctures of type $A$ of $D$}\}   \\
&& + \mu(D) -\sum_{\substack{p \in\{ \textrm{punctures}\\ \textrm{of type $B$ of $D$}\} }} \mu_{CZ}(p)  \\
&=& - \#\{\textrm{punctures of type $A$ of $D$}\}   \\
&& + d\left(\mu(D')-\sum_{\substack{p \in\{   \textrm{punctures}\\ \textrm{of type $B$ of $D'$}\} }} \mu_{CZ}(p)
\right) \\
&\geq& d\,\ind_\r(D')\;,
\end{eqnarray*}
since $\mu_{CZ}(p) = 2$ if $p$ is a puncture of type $A$ and $D \subset X\setminus U_0$, see Proposition~\ref{prop.metric.choice} and \cite[Proposition~5.2]{Bour}.

Again, the moduli space containing $D$ is of the expected dimension $\ind_\r(D)$. From this follows that the curve $C$ depends on

$$
\sum_{\substack{D \in\{   \textrm{components}\\ \textrm{of $C$}\} }} \ind_\r(D) = \mu(C) + \#\{\textrm{nodes of type $A$ of $C$}\}
$$
degrees of freedom.

Now, at each node of type $A$ of $C$, the two adjacent components of $C$ have to converge to the same geodesic of type $A$, which belongs to a one-dimensional space. Likewise, $C$ must contain the point $x$ and represent $p_k$. These constraints require
$$
 \#\{\textrm{nodes of type $A$ of $C$}\} + 4 + 2(k-2) = \mu(C) + \#\{\textrm{nodes of type $A$ of $C$}\}
 $$
degrees of freedom since by Definition~\ref{dfn.uniruled}, $\mu(C) = 4 + 2(k-2)$. As a consequence, all the components of $C$ are rigid. We deduce in particular that only one of the components of $C\cap U_0$ contains punctures of type $A$, the one containing $x$, and that this component has exactly two such punctures. Furthermore, if $C$ were not string-like, it would have a component isomorphic to $\c$ in $X\setminus U_0$, rigid, and converging to a geodesic of type $B$. The double covering of this component branched at one point to which we add the trivial cylinder of $U_0$ over the type $B$ geodesic  would provide a symplectic disc of   Maslov index zero with boundary on $L$, which is impossible. Hence the result.
\end{proof}

\begin{prop}\label{prop.nodal.curves}
Keeping Notation~\ref{nota.gene}, we assume that $(X,\omega)$ does not contain any symplectic disc 
of Maslov index zero whose boundary on $L$ does not vanish in $H_1 (L ; \q)$. Changing $S_0$ if necessary, the generic $S_0$-singular almost-complex structure $J_0$ can be chosen such that all rational $J_0$-holomorphic curves of $(X,\omega)$ representing $(A,p_k)$ satisfy the conditions of Lemma~\ref{lem.nodal.curves}.
\end{prop}

Hence, for all curves given by Proposition~\ref{prop.nodal.curves}, closed geodesics associated to their nodes are either of type $B$, or of type $A$ with Morse-Bott index one.

\begin{proof}
Let $C$ be a rational $J_0$-holomorphic curve of $(X,\omega)$ representing $(A,p_k)$. From Stokes' formula, the total length of the closed geodesics associated to the nodes of $C$ equals the energy of $C\cap U_0$, and thus gets bounded from above by the total energy $\int_A\omega$. Let $l_0$ be the length of the shortest closed geodesic of $(L,g_0)$, the number of punctures of $C$ is bounded from above by the quotient $\int_A\omega/l_0 =: N_0$. 

Let $\Pi_0$ be the finite set of homotopy classes of $L$ realized by closed geodesics of type $A$ and length $\leq \int_A\omega$. From Proposition~\ref{prop.metric.choice}, there exists a metric $g_1$ on $L$ such that no element of $\Pi_0$ gets realized by a closed geodesic of type $C$ for $g_1$ and
such that any closed geodesic of type $A$ which realizes an element of $\Pi_0$ is of Morse-Bott index~$1$.
 Without loss of generality, we can assume that the  Weinstein neighborhood $U_1$ of $L$ isometric to the unitary cotangent ball bundle of $(L,g_1)$ is strictly included in the interior of $U_0$. Let $S_1$
 be the boundary of $U_1$ and let us assume that $J_0$ is $(S_0\cup S_1)$-singular. Again, the total length of closed geodesics of $(L,g_1)$ associated to the nodes of $C\cap S_1$ is  bounded from above by $\int_A\omega$, and the set $\Pi_1$ of homotopy classes of $L$ realized by closed geodesics of type $A$ of length $\leq \int_A\omega$ for $g_1$ is finite. Proposition~\ref{prop.metric.choice} gives a metric $g_2$ on $L$ having, with respect to $\Pi_1$, the same properties as $g_1$ with respect to $\Pi_0$. We construct in this way a finite number of $\sol$-metrics $g_{N_0},\dots,g_0$ on $L$ which induce a finite number of  Weinstein neighborhoods $U_{N_0}\subset \dots \subset U_0$, of respective boundaries $S_{N_0},\dots,S_0$. Denote by $S$ the union $S_{N_0}\cup \dots \cup S_0$ and assume that $J_0$ is generic $S$-singular. Then, the pair $(S_{N_0},J_0)$ fits. That is, replacing $S_0$ by $S_{N_0}$ and $J_0$ by a $S_{N_0}$-singular structure close to a generic $S$-singular almost-complex structure, all rational $J_0$-holomorphic curves which represent $(A,p_k)$ satisfy the conditions of Lemma~\ref{lem.nodal.curves}.

Indeed, let $C$ be such a curve. The combinatorial type of $C\cap U_0$ is encoded by a forest whose leaves correspond to nodes of $C\cap S_0$, there are at most $N_0$ such leaves. Let $t$ be the number of trees of that forest, $\#\mathcal{A}$ the number of edges, $\#\mathcal{S}$ the number of vertices and for every vertex $s$, $v(s)$ be the valence of that vertex. The Euler formula gives the relation $\#\mathcal{S} - \#\mathcal{A} = t$, while $\#\mathcal{A} = \frac 12 \sum_{s\in \mathcal{S}}v(s)$, where $\mathcal{S}$ is the set of vertices. Hence, $t = \sum_{s\in \mathcal{S}}\left(1- \frac 12 v(s)\right)$ and we deduce the relation
$$
\frac 12 \#\bigl\{s\in \mathcal{S}\st v(s) \geq 3\bigr\} < t + \frac 12 \sum_{v(s)\geq 3}\left(v(s)-2\right) \leq \frac 12 N_0\;.
$$

There exists therefore $0\leq i < N_0$ such that $C \cap (U_i\setminus U_{i+1})$ contains only cylinders, encoded by bivalent vertices. All nodes of $C\cap S_{i+1}$ of type $A$ or $C$ thus correspond to closed geodesics homotopic to the ones associated to nodes of type $A$ or $C$ of $C\cap S_i$;  they are homotopic to $\Pi_i$. By construction of $g_{i+1}$, this implies that these nodes are of type $A$ and that the Morse-Bott indices of those geodesics all equal $1$. Lemma~\ref{lem.nodal.curves} applies to $(S_{i+1},J_0)$ and implies with the compactness Theorem \cite{BEHWZ} that all components of $C\cap U_{i+1}$ are cylinders. Again, by construction of the metrics $g_j$, $j\geq i+1$, we deduce that nodes of $C\cap S_{N_0}$ of type $A$ or $C$ are of type $A$ and Morse-Bott index one. Hence the result.
\end{proof}
\subsection{Proof of Theorem~\ref{thm.main}}

Let us assume that the manifold $(X,\omega)$ does not contain any symplectic disc of Maslov index zero whose boundary on $L$ does not vanish in $H_1(L,\q)$ and let us adopt Notation~\ref{nota.gene}. For every generic almost-complex structure $J$ of  $(X,\omega)$, we denote by $\mathcal{M}_{0,k+1}^{A,p_k}(H;J)$ the moduli space of rational $J$-holomorphic curves of $X$, homologous to $A$, conformal to $p_k$, which have $k+1$ marked points $x_0,\dots,x_k$ such that $x_i \in H_i$ for $1 \leq i \leq k$. We denote by $eval_0 \colon \mathcal{M}_{0,k+1}^{A,p_k}(H;J) \to X$ the evaluation map at $x_0$, its degree $\langle [pt]_k;[pt],\omega^k \rangle^X_A$ is  nontrivial by assumption.

Denote by $\mathcal{M}_{0,k+1}^{A,p_k}(H,L;J) = eval_0^{-1}(L)$, and by abuse 
$$
eval_0 \colon \mathcal{M}_{0,k+1}^{A,p_k}(H,L;J)  \to L
$$
the induced evaluation map. Its degree remains $\langle [pt]_k;[pt],\omega^k \rangle^X_A$ and thus nonzero.

From the compactness Theorem \cite{BEHWZ} in symplectic field theory, when $J$ converges to $J_0$, the space $\mathcal{M}_{0,k+1}^{A,p_k}(H,L;J)$ degenerates to a moduli space of string-like curves given by Proposition~\ref{prop.nodal.curves}. The unique non rigid component of any of these curves being a cylinder of $U_0$ which converges to a geodesic of type $A$. Every geodesic of type $A$ belongs to a $1$-parameter compact family. Let $A$ be such a family. We denote by $\widetilde L$ the infinite cyclic covering of $L$ associated to the projection $L \to B$, and by $\widetilde U_0$ the associated infinite cyclic covering of $U_0$. Let $\widetilde A$ be a lift of $A$ in $\widetilde U_0$. Then, every cylinder of $U_0$ asymptotic to an element of $A$ uniquely lifts in $\widetilde U_0$  to a cylinder asymptotic to an element of $\widetilde A$. Hence, if $\mathcal{M}$ is a compact family of cylinders of $U_0$ asymptotic to an element of $A$ and with a marked point in $L$, and if $ev \colon \mathcal{M} \to L$ is the associated evaluation map, this map lifts as $\widetilde{ev} \colon \mathcal{M} \to \widetilde L$ such that  the following diagram commutes. 
$$
\xymatrix{
   \mathcal{M}
   \ar[rrd]_{ev}\ar[rr]^{\widetilde{ev}}&&\widetilde L\ar[d]\\
   && L 
   }
 $$
We deduce that when $J$ is sufficiently close to $J_0$, the map $eval_0$ has a lift $\widetilde{eval_0} \colon \mathcal{M}_{0,k+1}^{A,p_k}(H,L;J) \to \widetilde L$ such that the diagram
$$
\xymatrix{
   \mathcal{M}_{0,k+1}^{A,p_k}(H,L;J)
   \ar[rrd]_-{eval_0}\ar[rr]^-{\widetilde{eval_0}}&&\widetilde L\ar[d]\\
   && L 
   }
 $$
 commutes.
This forces the degree of $eval_0$ to vanish and thus contradicts the hypothesis. \hfill{\qed}

\subsection{Final remarks}\label{rems}

\begin{enumerate}

\item \label{rem.difficulty}
If we do not assume that $(X,\omega)$ contains no symplectic disc of Maslov index zero whose boundary on $L$  does not vanish in $H_1(L,\q)$, then the irreducible component of a $J_0$-holomorphic curve homologous to $A$ containing $x \in U_0$ can be isomorphic to a sphere with more than two punctures, the additional punctures corresponding to geodesics of type $B$. Such components do not lift to $\widetilde U_0$, so that the argument used in the proof of Theorem~\ref{thm.main} to prove the vanishing of the degree of the evaluation map does not hold anymore. Furthermore, the counting of the number of such curves in $U_0 \cong T^*L$ passing through $x$ depends on the almost-complex structure $J_0$, or rather on the $CR$-structure on $\partial U_0$. We could not work out this case.

\item 
The minimal model program  applied to a uniruled three-dimensional projective manifold $X$ defined over $\r$ provides either a Fano variety, a Del Pezzo fibration over a curve or a conic bundle over a surface, all defined over $\r$. At least when the real locus $X(\r)$ is orientable, Koll\'ar proved \cite[Theorem~1.2]{koII} that if it contains a $\sol$ connected component, then so does the topological normalization of the real locus of its minimal model. Now, Koll\'ar proved \cite[Theorem~1.1]{koIII} that no conic bundle contains such a $\sol$ component while Corollary~\ref{cor.dp} together with \cite[\S 6.3]{koIV} proves that the same holds for Del Pezzo fibrations. Likewise Corollary~\ref{cor.fano} proves that there is no $\sol$ torus bundle in smooth Fano manifolds. It is possible to extend this result to the singular Fano varieties with only  real terminal singularities which might appear in this process. Indeed, Koll\'ar proved \cite[Theorem~1.10]{koII} that such singularities should be hypersuface singularities and these singularities can be symplectically smoothed. This way we get a symplectic deformation of the singular Fano variety together with a $\sol$ torus bundle Lagrangian submanifold. There cannot be a symplectic disc of vanishing Maslov index and boundary on this Lagrangian submanifold, since such a disc could be pushed away from the vanishing cycle of the singularity and thus would already exists in the singular Fano variety away from the singularity. This is impossible. In order to prove the non-existence of $\sol$ torus bundle component in a real uniruled projective three-fold with orientable real locus, it only remains to treat singular Fano varieties with complex conjugated singularities. The latter may be quotient singularities and we do not see right now  simple arguments to treat and include this case in the present paper.
 
\item
Koll\'ar points out that his results on the real MMP remain valid when the real locus of the manifold contains no two-sided $\r\p^2$, one-sided two-torus or one-sided Klein bottle with nonorientable neighborhood, see \cite[Condition~1.7] {koII}

\begin{lem}
Let $L$ be a $\sol$ closed three-dimensional manifold. Then $L$ contains no embedded $\r\p^2$, no embedded one-sided two-torus, and no embedded one-sided Klein bottle with nonorientable neighborhood.
\end{lem}

\begin{proof}
This lemma follows from Koll\'ar-Kapovich \cite[Theorem~12.2]{koII}. Here follows a direct proof. From Theorem~\ref{thm.classif}, nonorientable $\sol$ manifolds are nonorientable torus bundles.
\begin{enumerate}
\item 
Any $\r\p^2$ in $L$ would lift to a $\r\p^2$ in the universal covering $\r^3$ of $L$ since $\pi_1(L)$ has no order two element. And $\r^3$ contains no nonorientable hypersurface (compare \cite[Lemma~12.3]{koII}). 

\item
Let $i \colon K \hookrightarrow L$ be a one-sided Klein bottle with nonorientable neighborhood. The induced morphism $\pi_1(K) \stackrel{i_*'}{\hookrightarrow} H_1(L)/\tors \cong \z$ factorizes through an injective morphism $H_1(K)/\tors \cong \z$, since otherwise it would vanish and $K$ would lift to the infinite cyclic covering of $L$ and then would have an orientable neighborhood. Let us write $\pi_1(K) = \langle a,b \st aba^{-1}b=1 \rangle$, where $a$ generates $H_1(K)/\tors$ and $b^2$ generates $[\pi_1(K),\pi_1(K)]$. Then $i_*'(b^2)=0$ so that $i_*'(b)=0$. We deduce that $i^*w_1(L)(a) = 1$ and $i^*w_1(L)(b) = 0$ so that $i^*w_1(L) = w_1(K)$, which contradicts the fact that $K$ should be one-sided. 

\item
Let $j \colon T \hookrightarrow L$ be a one-sided torus. The image of the induced morphism $j_*' \colon \pi_1(T) \hookrightarrow H_1(L)/\tors \cong \z$ is a subgroup $N\z$ with $N$ odd, since as before $T$ would lift to a torus in the $N^{th}$ cyclic covering of $L$ which would be orientable. Let us write $\pi_1(T) = \langle a,b \st aba^{-1}b^{-1}=1 \rangle$, where $j_*'(a)$ generates $\mathrm{Im} j_*'$ and $b$ generates $\ker j_*'$. Then $j_*(b) = j_* (aba^{-1})=A^N(j_*(b))$ where $A$ is the hyperbolic monodromy map of the bundle, so that $j_*(b) = 0$ where $j_* \colon \pi_1(T) \to \pi_1(L)$. Hence $T$ lifts to the covering $\widehat L$ of $L$ associated to $\mathrm{Im} j_*'$, this  is the plane bundle with monodromy $A$. Let $s$ be the boundary of a neighborhood of the zero section of the normal bundle of $T$ restricted to $a$. Then $s$ is disjoint from $T$ and $0\ne b \in \pi_1(\widehat L\setminus s)$. Indeed, if $D$ is a disc of $\widehat L$ with boundary $b$, then $D\cdot s = 2 \mathring{D}\cdot a + a \cdot b \equiv 1 \mod 2$. Now  $\pi_1(\widehat L\setminus s) = \langle t_1,t_2,a \st at_1a^{-1} = t_2^{-1}, at_2a^{-1} = t_1^{-1}\rangle$ where $\langle t_1,t_2\rangle$ generate the free fundamental group of the fiber $\widehat L \to S^1$. The element $b$ is a word in $t_1,t_2$ which for the same reason as before satisfies $A(b) = b$, this is impossible.
\end{enumerate}

\end{proof}
\end{enumerate}

\bibliography{solmanifolds}

\begin{thebibliography}{10}

\bibitem{AkSal}
M.~Akveld and D.~Salamon.
\newblock Loops of {L}agrangian submanifolds and pseudoholomorphic discs.
\newblock {\em Geom. Funct. Anal.}, 11(4):609--650, 2001.

\bibitem{Bour}
F.~Bourgeois.
\newblock A {M}orse-{B}ott approach to {C}ontact {H}omology.
\newblock {\em Ph.D dissertation, Stanford University}, 2002.

\bibitem{BEHWZ}
F.~Bourgeois, Y.~Eliashberg, H.~Hofer, K.~Wysocki, and E.~Zehnder.
\newblock Compactness results in symplectic field theory.
\newblock {\em Geom. Topol.}, 7:799--888 (electronic), 2003.

\bibitem{com}
A.~Comessatti.
\newblock Sulla connessione delle superfizie razionali reali.
\newblock {\em Annali di Math.}, 23(3):215--283, 1914.

\bibitem{EGH}
Y.~Eliashberg, A.~Givental, and H.~Hofer.
\newblock Introduction to symplectic field theory.
\newblock {\em Geom. Funct. Anal.}, (Special Volume, Part II):560--673, 2000.
\newblock GAFA 2000 (Tel Aviv, 1999).

\bibitem{Fuk}
K.~Fukaya.
\newblock Application of {F}loer homology of {L}angrangian submanifolds to
  symplectic topology.
\newblock In {\em Morse theoretic methods in nonlinear analysis and in
  symplectic topology}, volume 217 of {\em NATO Sci. Ser. II Math. Phys.
  Chem.}, pages 231--276. Springer, Dordrecht, 2006.

\bibitem{HWZ1}
H.~Hofer, K.~Wysocki, and E.~Zehnder.
\newblock Properties of pseudoholomorphic curves in symplectisations. {I}.
  {A}symptotics.
\newblock {\em Ann. Inst. H. Poincar\'e Anal. Non Lin\'eaire}, 13(3):337--379,
  1996.

\bibitem{HWZIII}
H.~Hofer, K.~Wysocki, and E.~Zehnder.
\newblock Properties of pseudoholomorphic curves in symplectizations. {III}.
  {F}redholm theory.
\newblock In {\em Topics in nonlinear analysis}, volume~35 of {\em Progr.
  Nonlinear Differential Equations Appl.}, pages 381--475. Birkh\"auser, Basel,
  1999.

\bibitem{HWZfinite}
H.~Hofer, K.~Wysocki, and E.~Zehnder.
\newblock Finite energy foliations of tight three-spheres and {H}amiltonian
  dynamics.
\newblock {\em Ann. of Math. (2)}, 157(1):125--255, 2003.

\bibitem{hu-li-ruan}
J.~Hu, T.-J. Li, and Y.~Ruan.
\newblock Birational cobordism invariance of uniruled symplectic manifolds.
\newblock {\em Invent. Math.}, 172(2):231--275, 2008.

\bibitem{HuLal}
S.~Hu and F.~Lalonde.
\newblock Homological {L}agrangian monodromy.
\newblock {\em Preprint arXiv 0912.1325}, 2009.

\bibitem{hm2}
J.~Huisman and F.~Mangolte.
\newblock Every connected sum of lens spaces is a real component of a uniruled
  algebraic variety.
\newblock {\em Ann. Inst. Fourier (Grenoble)}, 55(7):2475--2487, 2005.

\bibitem{hm1}
J.~Huisman and F.~Mangolte.
\newblock Every orientable {S}eifert 3-manifold is a real component of a
  uniruled algebraic variety.
\newblock {\em Topology}, 44(1):63--71, 2005.

\bibitem{Keel}
S.~Keel.
\newblock Intersection theory of moduli space of stable {$n$}-pointed curves of
  genus zero.
\newblock {\em Trans. Amer. Math. Soc.}, 330(2):545--574, 1992.

\bibitem{KhBourbaki}
V.~Kharlamov.
\newblock Vari\'et\'es de {F}ano r\'eelles (d'apr\`es {C}. {V}iterbo).
\newblock {\em Ast\'erisque}, (276):189--206, 2002.
\newblock S{\'e}minaire Bourbaki, Vol. 1999/2000.

\bibitem{kollar.sp.uni}
J.~Koll{\'a}r.
\newblock Low degree polynomial equations: arithmetic, geometry and topology.
\newblock In {\em European {C}ongress of {M}athematics, {V}ol.\ {I}
  ({B}udapest, 1996)}, volume 168 of {\em Progr. Math.}, pages 255--288.
  Birkh\"auser, Basel, 1998.

\bibitem{ko-nash}
J.~Koll{\'a}r.
\newblock The {N}ash conjecture for threefolds.
\newblock {\em Electron. Res. Announc. Amer. Math. Soc.}, 4:63--73
  (electronic), 1998.

\bibitem{koI}
J.~Koll{\'a}r.
\newblock Real algebraic threefolds. {I}. {T}erminal singularities.
\newblock {\em Collect. Math.}, 49(2-3):335--360, 1998.
\newblock Dedicated to the memory of Fernando Serrano.

\bibitem{koII}
J.~Koll{\'a}r.
\newblock Real algebraic threefolds. {II}. {M}inimal model program.
\newblock {\em J. Amer. Math. Soc.}, 12(1):33--83, 1999.

\bibitem{koIII}
J.~Koll{\'a}r.
\newblock Real algebraic threefolds. {III}. {C}onic bundles.
\newblock {\em J. Math. Sci. (New York)}, 94(1):996--1020, 1999.
\newblock Algebraic geometry, 9.

\bibitem{koIV}
J.~Koll{\'a}r.
\newblock Real algebraic threefolds. {IV}. {D}el {P}ezzo fibrations.
\newblock In {\em Complex analysis and algebraic geometry}, pages 317--346. de
  Gruyter, Berlin, 2000.

\bibitem{KontManin}
M.~Kontsevich and Y.~Manin.
\newblock Gromov-{W}itten classes, quantum cohomology, and enumerative
  geometry.
\newblock {\em Comm. Math. Phys.}, 164(3):525--562, 1994.

\bibitem{McDSal}
D.~McDuff and D.~Salamon.
\newblock {\em Introduction to symplectic topology}.
\newblock Oxford Mathematical Monographs. The Clarendon Press Oxford University
  Press, New York, second edition, 1998.

\bibitem{morimoto}
K.~Morimoto.
\newblock Some orientable {$3$}-manifolds containing {K}lein bottles.
\newblock {\em Kobe J. Math.}, 2(1):37--44, 1985.

\bibitem{scott}
P.~Scott.
\newblock The geometries of {$3$}-manifolds.
\newblock {\em Bull. London Math. Soc.}, 15(5):401--487, 1983.

\bibitem{tro}
M.~Troyanov.
\newblock L'horizon de {${\rm SOL}$}.
\newblock {\em Exposition. Math.}, 16(5):441--479, 1998.

\bibitem{V}
C.~Viterbo.
\newblock A new obstruction to embedding {L}agrangian tori.
\newblock {\em Invent. Math.}, 100(2):301--320, 1990.

\bibitem{sftVit}
C.~Viterbo.
\newblock Symplectic real algebraic geometry.
\newblock {\em Unpublished}, 1999.

\bibitem{weinstein}
A.~Weinstein.
\newblock Symplectic manifolds and their {L}agrangian submanifolds.
\newblock {\em Advances in Math.}, 6:329--346, 1971.

\bibitem{wels}
J.-Y. Welschinger.
\newblock Effective classes and {L}agrangian tori in symplectic four-manifolds.
\newblock {\em J. Symplectic Geom.}, 5(1):9--18, 2007.

\bibitem{WelsOpt}
J.-Y. Welschinger.
\newblock Optimalit{\'e}, congruences et calculs d'invariants des
  vari{\'e}t{\'e}s symplectiques r{\'e}elles de dimension quatre.
\newblock {\em Preprint math.arXiv:0707.4317}, 2007.

\bibitem{welsch}
J.-Y. Welschinger.
\newblock Open strings, {L}agrangian conductors and {F}loer functor.
\newblock {\em Preprint math.arXiv:0812.0276}, 2008.

\end{thebibliography}
\bibliographystyle{abbrv}

\noindent
Universit\'e de Savoie ; Laboratoire de math\'ematiques (LAMA)

\medskip

\noindent
Universit\'e de Lyon ; CNRS ;
Universit\'e Lyon~1 ; Institut Camille Jordan


\end{document}